\documentclass[12pt]{article}
\usepackage{amsfonts}
\usepackage{amssymb, amscd, amsthm}
\usepackage{latexsym}
\usepackage{amsmath}
\usepackage[all]{xy}
\usepackage[dvips]{graphicx}
\usepackage{mathrsfs}
\usepackage{fancyhdr}

\newtheorem{lemma}[subsection]{Lemma}
\newtheorem{thm}[subsection]{Theorem}
\newtheorem{prop}[subsection]{Proposition}
\newtheorem{rem}[subsection]{Remark}
\newtheorem{coro}[subsection]{Corollary}

\newcommand{\ra}{\rightarrow}
\newcommand{\mhu}{M_X^H(u)}

\newcommand{\mchi}{M_{\chi}}
\newcommand{\uchi}{u_{\chi}}
\newcommand{\lu}{\lambda_u}

\newcommand{\lcn}{\lambda_{c^r_n}}
\newcommand{\mo}{\mathcal{O}}
\newcommand{\mf}{\mathcal{F}}

\newcommand{\mc}{\mathcal{C}}

\newcommand{\z}{\Theta}
\newcommand{\hk}{\textbf{K}}
\newcommand{\ha}{\textbf{A}}

\newcommand{\ls}{|L|}
\newcommand{\pone}{\mathbb{P}^1}

\begin{document}
\fontsize{12pt}{14pt} \textwidth=14cm \textheight=21 cm
\numberwithin{equation}{section}
\title{Estimates of sections of determinant line bundles on Moduli spaces of pure sheaves on algebraic surfaces}
\author{Yao YUAN}
\date{\small\textsc SISSA, Via Bonomea 265, 34136, Trieste, ITALY
\\ yuayao@gmail.com}
\maketitle
\begin{flushleft}{\textbf{Abstract:}}
Let $X$ be any smooth simply connected projective surface.  We consider some moduli space of pure sheaves of dimension one on $X$,  i.e. $\mhu$ with $u=(0,L,\chi(u)=0)$ and $L$ an effective line bundle on $X$,  together with a series of determinant line bundles associated to $r[\mo_X]-n[\mo_{pt}]$ in Grothendieck group of $X$.  Let $g_L$ denote the arithmetic genus of curves in the linear system $\ls$.  For $g_L\leq2$,  we give a upper bound of the dimensions of sections of these line bundles by restricting them to a generic projective line in $\ls$.  Our result gives,  together with G\"ottsche's computation,  a first step of a check for the strange duality for some cases for $X$ a rational surface.   
\end{flushleft}
\section{Introduction.}
let $X$ be a smooth complex projective surface with $H$ an ample divisor,  and $u$ and $c^r_n$ two elements in the Grothendiek group $\hk(X)$ of $X$ which are specified as $u=(0,L,\chi(u)=0)$ for $L$ an effective line bundle on $X$,  and $c_n^r=r[\mathcal{O}_X]-n[\mathcal{O}_{pt}]$ where $\mo_{pt}$ is the skyscraper sheaf supported at a point in $X.$  Denote $\mhu$ (resp. $M^H_X(c_n^r)$) the moduli space of semistable sheaves with respect to $H$ on $X$ of class $u$ (resp. $c_n^r$).  There is a so-called determinant line bundle $\lcn$ (resp. $\lu$) on $\mhu$ (resp. $M^H_X(c_n^r)$) associated to $c_n^r$ (resp. $u$) (See \cite{dan} Chapter 8 for more details).   It is conjectured by Strange Duality that there is a natural isomorphism between the following two spaces (see \cite{mofir} for more details)
\begin{equation}\label{sdm}D:H^0(\mhu,\lcn)^{\vee}\ra H^0(M_X^H(c_n^r),\lu).\end{equation}
We are concerned on the numerical version of the conjecture.  In other words,  we would like to check the following equality
\begin{equation}\label{numcon}h^0(\mhu,\lcn)=h^0(M^H_X(c_n^r),\lu).
\end{equation}

In \cite{yuan} for $X=\mathbb{P}^2$ or $\mathbb{P}(\mo_{\pone}\oplus\mo_{\pone}(-e))$ with $e=0,1$ and $L=2G+aF$ with $2e\leq a\leq e+3$ where $F$ is the fiber class and $G$ is the section such that $G.G=-e$,  we have computed the generating function 
\begin{equation}\label{generating}Z^r(t)=\sum_{n\geq0}h^0(\mhu,\lcn)t^n,\end{equation}
for all $r\geq1$.  Moreover when $r=2$,  the result matches G\"ottsche's computation on the rank 2 sheaves side and gives a numerical check of Strange Duality for these cases (See \cite{yuan} Corollary 4.4.2 and Corollary 4.5.3). 

In this paper we consider more general cases.  We ask $X$ to be any smooth simply connected projective surface over the complex number $\mathbb{C}$.  Let $K$ be the canonical divisor of $X$.  Let $\ls$ be the linear system associated to the line bundle $L$ and $l$ the dimension of $\ls$.  Let $g_L$ be the arithmetic genus of curves in $\ls$.  For any two line bundles $L$ and $L'$,  we denote $L.L'$ to be the intersection number of their divisors;   and moreover we write $L'\leq L$ if $L\otimes L'^{-1}$ is an effective line bundle,  i.e. $h^0(L\otimes L'^{-1})\neq0;$  and write $L'<L$ if $L'\leq L$ and $L'\neq L.$  We state two assumptions on $L$ as follows which are all the assumptions we need

$(\ha'_1)$ $L.K<0;$

$(\ha'_2)$ For any $0<L',L''<L$ with $L'+L''=L,$  we have $l'+l''\leq l-2$ where $l'=dim~|L'|$ and $l''=dim~|L''|.$  

Since we deal with more general cases,  the techniques we used in \cite{yuan} to obtain the normality and irreducibility of the Moduli space $\mhu$ and the dualizing sheaf on $\mhu$ don't work any more.  We thus lose many good properties of the moduli spaces,  but anyway we still have some results providing an estimate for the dimension of sections of $\lcn$ on $\mhu$.  We have obtained in this paper the following three theorems:
\begin{thm}\label{thmone}Let $X$ be simply connected and let $L$ satisfy $(\ha'_1)$ and $(\ha'_2)$. 
Then we have for all $n\geq0$
\[h^0(M(c^1_n),\lu)\geq h^0(M(u),\lambda_{c^1_n}).\]
Moreover for any fixed $r$,  once the strict inequality holds for $n=n_0$,  it holds for all $n\geq n_0.$
\end{thm}

Denote
\[Y^r_{g_L=1}(t)=\sum_{n\geq0}y^{r}_{n,g_L=1}t^n=\frac{1+t^{2}+t^{3}+\ldots+t^{r}}{(1-t)^2};\]
and let $y_{n,g_L=1}^r=0$ for all $n<0$.  Then we have 
\begin{thm}\label{thmtwo}Let $X$ be a smooth simply connected projective surface and $L$ satisfy $(\ha'_1)$ and $(\ha'_2)$ with $g_L=1$.  
Then we have for all $n\in \mathbb{Z}$ and $r\geq1$,
\[y^r_{n,g_L=1}\geq h^0(M(u),\lambda_{c^r_n}).\]
Moreover for any fixed $r$,  once the strict inequality holds for $n=n_0$,  it holds for all $n\geq n_0.$
\end{thm}

Let $Y^1_{g_L=2}=\sum_{n}y_{n,g_L=2}^1t^n=\frac1{(1-t)^2}$ and for $r\geq2$
\[Y^r_{g_L=2}(t)=\sum_{n}y_{n,g_L=2}^rt^n=\frac{1+3t^{2}+\sum_{i=3}^r ((i+1)t^{i}+(i-2)t^{i+1})}{(1-t)^{l+1}}.\]
Let $y_{n,g_L=2}^r=0$ for all $n<0$.  Then we have
\begin{thm}\label{thmthree}Let $X$ be a a smooth simply connected projective surface and $L$ satisfy $(\ha'_1)$ and $(\ha'_2)$ with $g_L=2$ and $dim~\ls\geq3$.  
Then we have for all $n\in \mathbb{Z}$ and $r\geq1$,
\[y^r_{n,g_L=2}\geq h^0(M(u),\lambda_{c^r_n}).\]
Moreover for any fixed $r$,  once the strict inequality holds for $n=n_0$,  it holds for all $n\geq n_0.$
\end{thm}
\begin{rem}Fix $r=2$.  G\"ottsche's results for rational ruled surfaces together with his blow-up formulas give many examples for $X$ a rational surface,  in which $L$ satisfies $(\ha'_1)$ and $(\ha'_2)$ with $g_L=1$ or $g_L=2$ and $l\geq3$,  and also the following equalities hold under some suitable polarization (a change of the polarization may give a difference of a polynomial)  
\[\sum_{n\geq0}\chi(M(c_n^2),\lambda_L)t^n=\frac{1+t^2}{(1+t)^{l+1}}=Y^2_{g_L=1}(t),~if~g_L=1;\]
\[\sum_{n\geq0}\chi(M(c_n^2),\lambda_L)t^n=\frac{1+3t^2}{(1+t)^{l+1}}=Y^2_{g_L=2}(t),~if~g_L=2.\]
Hence we have for these cases under a suitable polarization for all $n\geq0$
\[\chi(M(c^2_n),\lu)\geq h^0(M(u),\lambda_{c^2_n}).\]
In particular (under any polarization) for $n\gg0$,  we have
\[\chi(M(c^2_n),\lu)=h^0(M(c^2_n),\lu)\geq h^0(M(u),\lambda_{c^2_n}).\]
\end{rem}

The main idea to prove these three theorems is to restrict $\z^r$ to intersections of pull back of hyperplanes in $\ls$ until finally we reach a generic projective line $T$ in $\ls$.  We then compute the splitting type of $\pi_{*}(\z^r|_{\pi^{-1}(T)})$ on $T$.  We prove Theorem \ref{thmone} in Section 4,  Theorem \ref{thmtwo} in Section 5.  The proof of Theorem \ref{thmthree} is the most complicated one among the three and is done in Section 6.  Also in Section 6 we obtain a corollary (Corollary \ref{jaco}) in the theory of compactified Jacobian of integral curves with planar singularities. 
\section{Notations.}
Let $\uchi$ be an element in $\hk(X)$ given by $\uchi=(0,L,\chi(\uchi)=\chi)$,  and $\mchi$ the  moduli space of semistable sheaves (w.r.t. $H$) of class $\uchi$ on $X$.  Denote $\mchi^s$ the stable locus of $\mchi$.  Notice that when $g.c.d(\chi,L.H)=1,$  $\mchi=\mchi^s.$ 

Let $\ls^{IC}$ be the open subset of $\ls$ consisting of points corresponding to integral curves.  By $(\ha'_2)$,  we have $\ls-\ls^{IC}$ is of codimension $\geq 2$ in $\ls$.  

There is a projection $\pi_{\chi}:\mchi\ra\ls$ which is defined by sending every sheaf to its schematic support.  $\pi_{\chi}$ is a morphism according to Proposition 3.0.2 in \cite{yuan}.  $(\ha'_1)$ implies that Ext$^2(\mf,\mf)=0$ for all $\mf$ semistable of class $\uchi$ that are supported on integral curves.  Hence by Lemma 4.2.3 in \cite{yuan} the moduli space $\mchi$ is smooth of dimension $g_L+l$ at the point $[\mf]$ if $\mf$ is supported on an integral curve,  i.e. $\pi_{\chi}([\mf])\in\ls^{IC}$. 

For $\chi=0$ we write $u$,  $M$,  $M^s$ and $\pi$ instead.  It is easy to see that $M$ does not depend on the polarization,  but $\mchi$ might for $\chi\neq0.$

We denote $\z$ and $\lambda_{pt}$ the determinant line bundles on $\mhu$ associated to $[\mo_X]$ and $[\mo_{pt}]$.  Hence we have $\lcn\simeq \z^{\otimes r}\otimes \lambda_{pt}^{\otimes -n}.$  We moreover ask $\mo_{pt}$ not to be supported at the base point of $\ls$,  then by Proposition 2.8 in \cite{le} we have that $\lambda_{pt}\simeq\pi^{*}\mo_{\ls}(-1)$.  Let $\z^r(n):=\z^r\otimes\pi^{*}\mo_{\ls}(n)$. 

\section{Restrict $\z^r$ to intersections of pull backs of hyperplanes in $\ls|$.}
 Choose $l-1$ generic points in $X$: $x_1,x_2,\ldots,x_{l-1}.$  For each $x_i$,  by asking the supporting curves of the sheaves to pass through it,  we can get an equation $f_i$ up to scalar in $|\pi_{\chi}^{*}\mo_{\ls}(1)|$.  Let $V_i$ be the divisor defined by $f_i$.  Since $x_1,\ldots,x_{l-1}$ are generic,  we let $V_i$ intersect each other transversally.  There is also a series of closed subschemes in $\ls$:  $P_1,P_2,\ldots,P_{l-1},$  where $P_{i}$ consists of curves passing through $x_1,\ldots,x_{i}.$  $P_{i}\simeq\mathbb{P}^{l-i}$ and $\pi_{\chi}^{-1}(P_{i})=\cap_{1\leq m\leq i}V_m$.  Let $T:=P_{l-1}$.  Then $T$ is a projective line in $\ls.$  

Because $\ls-\ls^{IC}$ is of codimension $\geq 2$ in $\ls$,  we can assume that $T\subset \ls^{IC}$.  We then have the following Cartesian diagram
\begin{equation}\label{lnode}
\xymatrix{\mchi^T\ar[r]^s\ar[d]_{\pi^T_{\chi}}&\mchi^{IC}\ar[r]^j\ar[d]^{\pi^{IC}_{\chi}}&\mchi\ar[d]^{\pi_{\chi}}\\
T\ar[r]^t&\ls^{IC}\ar[r]^i&\ls}
\end{equation}
$\mchi^{IC}$ is contained in the stable locus $\mchi^s$ and is smooth.  We can also assume that $\mchi^T$ is smooth since $|\pi_{\chi}^{*}\mo_{\ls}(1)|$ has no base point.

For $\chi=0,$  $\mchi=M,$  we have an exact sequence on $M:$
\begin{equation}\label{pone}
0\ra\pi_{\chi}^{*}\mo_{\ls}(-1)\ra\mo_{M}\ra\mo_{\pi^{-1}(P_1)}\ra0.
\end{equation}    

We then tensor (\ref{pone}) by $\z^r(n)$
\begin{equation}\label{ot}
\xymatrix{
0\ar[r]&\z^r(n-1)\ar[r]&\z^r(n)\ar[r]&\z^r(n)|_{\pi^{-1}(P_1)}\ar[r]&0.}
\end{equation}
Taking the global sections,  we have
\begin{equation}\label{otv}
0\ra H^0(\z^r(n-1))\ra H^0(\z^r(n))\ra H^0(\z^r(n)|_{\pi^{-1}(P_1)})\ra H^1(\z^r(n-1)).
\end{equation}

Sequence (\ref{otv}) implies that $h^0(\z^r(n))-h^0(\z^r(n-1))\leq h^0(\z^r(n)|_{\pi^{-1}(P_1)}).$

Denote $Z^r_i(t)=\sum_{n}h^0(M,\z^r(n)|_{\pi^{-1}(P_i)})t^n$ for all $i=1,\ldots,l-1.$  Notice that the sum are bounded from below for all $i$.  Hence we have 
\begin{equation}\label{sim}h^0(M,\z^r(n))\leq\sum_{m\leq n}h^0(\z^r(n)|_{\pi^{-1}(P_1)})\end{equation}
The inequality (\ref{sim}) will become an equality if $h^1(M,\z^r(n-1))=0$ for all $n$ such that $h^0(\pi^{-1}(P_1),\z^r(n)|_{\pi^{-1}(P_1)})\neq0.$  And once the strict inequality holds for $n=n_0$,  it holds for all $n\geq n_0.$  On the other hand we have 
\[\sum_{n}(\sum_{m\leq n}h^0(\z^r(n)|_{\pi^{-1}(P_1)}))t^n=\frac{Z^r_1(t)}{1-t}.\]   

Inductively for all $1\leq i\leq l-2,$   we have an exact sequence
\begin{equation}\label{itv}
0\ra\z^r(n-1)|_{\pi^{-1}(P_i)}\ra\z^r(n)|_{\pi^{-1}(P_i)}\ra\z^r(n)|_{\pi^{-1}(P_{i+1})}\ra0,
\end{equation} 
This implies that 
\begin{equation}\label{sime}h^0(M,\z^r(n)|_{\pi^{-1}(P_i)})\leq\sum_{m\leq n}h^0(\z^r(n)|_{\pi^{-1}(P_{i+1})})\end{equation}

Finally we come to the generic projective line $T=P_{l-1}$ in the linear system.  Define
\[\sum_{n}a^r_nt^n:=\frac{Z^r_{l-1}(t)}{(1-t)^{l+1}}.\]  
Then we have 
\begin{equation}\label{rgeq}h^0(M,\z^r(n))\leq a_n^r.\end{equation}

We will compute $Z^r_{l-1}(t)$ for $g_L=1,2$ in the next sections.  

\section{Moduli spaces over one dimensional linear systems.}
In this section,  we construct a new moduli space $\tilde{\mchi}$ over a one dimensional linear system $\tilde{\ls}$ on a surface $\tilde{X}$ obtained by blowing up points in $X$.  Then we show that $\tilde{\mchi}$ can be identified with $\mchi^T.$  The construction is as follows.

Choose $l-1$ generic points in $X$: $x_1,x_2,\ldots,x_{l-1};$  such that curves passing through all these $l-1$ points are integral curves (this is to say that the line $T$ defined by those points is contained in $\ls^{IC}$) and all of them except finitely many are smooth.  Moreover those curves are smooth at $x_1,x_2,\ldots,x_{l-1}$ (this is possible since the points are finitely many).   We then blow up all these $l-1$ points and get a new surface $\tilde{X}$ together with a projection $\rho:\tilde{X}\ra X.$  We have a new moduli space $\tilde{\mchi}=M_{\tilde{X}}(\tilde{u}_{\chi}),$  where $\tilde{u}_{\chi}=(0,\tilde{L}=\rho^{*}L-E_1-E_2-\ldots-E_{l-1},\chi)$ with the $E_i$ the exceptional divisors.  Notice that there is a natural closed embedding $\imath:|\tilde{L}|\ra\ls$ with its image $T.$  In particular for $\tilde{u}_{\chi}=\tilde{u}_{0}=:\tilde{u},$  we denote $\tilde{\z}$ the determinant line bundle on $\tilde{M}=\tilde{M}_0$ associated to the structure sheaf $\mo_{\tilde{X}}.$  Then we have the following proposition:

\begin{prop}\label{muse}
There is a morphism $\underline{f}:\tilde{\mchi}\ra \mchi,$  which factors through the embedding $j\circ s$ as in diagram (\ref{lnode}) and induces an isomorphism $f:\tilde{\mchi}\ra \mchi^T;$  and we have the Cartesian diagram as follows 
\begin{equation}\label{car}\xymatrix{
  \tilde{\mchi} \ar[d]_{\tilde{\pi}_{\chi}}  \ar[r]^{f}
                & \mchi^T\ar[d]_{\pi^T_{\chi}}\ar[r]^{j\circ s}&\mchi\ar[d]^{\pi_{\chi}}  \\
                |\tilde{L}|\ar[r]^{\imath}&T\ar[r]^{i\circ t} &\ls            }\end{equation}

And moreover for $\chi=0$ we have $\underline{f}_{*}\tilde{\z}^r\simeq (j\circ s)_{*}f_{*}\tilde{\z}^r$ and $f_{*}\tilde{\z}^r\simeq (j\circ s)^{*}\z^r.$
\end{prop}
\begin{proof}First we have two lemmas
\begin{lemma}\label{univ}There is a universal sheaf on $\tilde{X}\times \tilde{\mchi}.$  That is to say,   $\tilde{\mchi}$ is a fine moduli space.
\end{lemma}
\begin{proof}Let $\tilde{\Omega}_{\chi}$ be the open subscheme of the $Quot$-scheme and $\tilde{\phi}_{\chi}:\tilde{\Omega}_{\chi}\ra \tilde{\mchi}$ be the good quotient.  Since all curves in $|\tilde{L}|$ are irreducible and reduced,  all semistable sheaves in $\tilde{u}_{\chi}$ are stable and the morphism $\tilde{\phi}_{\chi}:\tilde{\Omega}_{\chi}\ra \tilde{\mchi}$ is a principal $G$-bundle,  with $G$ some reductive group.  There is a universal quotient $\tilde{\mathcal{E}}_{\chi}$ on $\tilde{X}\times \tilde{\Omega}_{\chi}$.  
\[\xymatrix{
  \tilde{\mathcal{E}}_{\chi}  \ar[r]
                & \tilde{X}\times \tilde{\Omega}_{\chi} \ar[ld]^{q} \ar[d]^{p_{\chi}}  \\
                \tilde{X} &\tilde{\Omega}_{\chi}             }\]
Let $A=det~R^{\bullet}p_{\chi }(\tilde{\mathcal{E}}_{\chi}\otimes q^{*}\mathcal{O}_{\tilde{X}}((1-\chi)E_1)).$  $A$ is a line bundle on $\tilde{\Omega}_{\chi}$ and carries a natural $G$-linearization of $Z$-weight $\chi((\tilde{\mathcal{E}}_{\chi})_y\otimes\mathcal{O}_{\tilde{X}}((1-\chi)E_1))$ for every closed point $y\in\tilde{\Omega}_{\chi}$.  Since $E_i.\tilde{L}=1$ and $(\tilde{\mathcal{E}}_{\chi})_y$ is of rank 0 and Euler characteristic $\chi$ for every $y,$  we have $\chi((\tilde{\mathcal{E}}_{\chi})_y\otimes\mo_{\tilde{X}}((1-\chi)E_1))=1$ which means $A$ is of $Z$-weight 1.  According to Proposition 4.6.2 and Theorem 4.6.5 in \cite{dan},  we have the lemma. 
\end{proof}

\begin{lemma}\label{integral}$\tilde{\pi}$ is flat and $\tilde{\mchi}$ is an integral scheme.
\end{lemma}
\begin{proof}Since curves in $|\tilde{L}|$ are reduced and irreducible and with at most planar singularities,  every fiber of $\tilde{\pi}$ is integral and of dimension $g.$  Hence $\tilde{\mchi}$ can not have more than one component because $|\tilde{L}|$ is just a projective line.  Then $\tilde{\pi}$ is flat because there is no component contained in any fiber.  $\tilde{\mchi}$ is reduced because all  fibers of $\tilde{\pi}$ are reduced and $|\tilde{L}|$ is reduced.
\end{proof}

Now let $\tilde{\mathcal{U}}_{\chi}$ be a universal sheaf on $\tilde{X}\times \tilde{\mchi}$.  Push it forward along $\rho\times id_{\tilde{\mchi}}$ and get a flat family $\mathcal{U}_{\chi}:=(\rho\times id_{\tilde{\mchi}})_{*}\tilde{\mathcal{U}}_{\chi}$ on $X\times \tilde{\mchi}.$  

Over every point $[\mathcal{F}]\in \tilde{\mchi},$  $\rho_{*}\mathcal{F}$ is a stable sheaf whose support is the push forward of the support of $\mathcal{F},$  hence $[\rho_{*}\mathcal{F}]\in \mchi^T.$  The flat family $\mathcal{U}_{\chi}$ induces a morphism $\underline{f}:\tilde{\mchi}\rightarrow \mchi,$  with its image contained in $\mchi^T.$

Since $\mchi^T$ is smooth hence normal and $\tilde{\mchi}$ is integral,  to prove that $f:\tilde{\mchi}\ra \mchi^T$ is an isomorphism,  it is enough to show that it is bijective.  The injectivity is because $\rho|_{C_{\mathcal{F}}}:C_{\mathcal{F}}\rightarrow C_{\rho_{*}\mathcal{F}}$ is an isomorphism,  where $C_{\mathcal{F}}$ is the supporting curve of $\mathcal{F}.$  To prove the surjectivity,  we need to show that $\forall [\mathcal{G}]\in \mchi^T,$ $\exists [\tilde{\mathcal{G}}]\in \tilde{\mchi}$ such that $\rho_{*}\tilde{\mathcal{G}}\simeq \mathcal{G}.$   Pull back $\mathcal{G}$ to get a sheaf on $\tilde{X}$ with support $C=C_{\rho^{*}\mathcal{G}}\in |\rho^{*}L|.$  On $\tilde{X}$ we have
\[0\rightarrow\mathcal{O}_{E_i}(-1)^{\oplus_{i=1}^{l-1}}\rightarrow\mathcal{O}_{C}\rightarrow\mathcal{O}_{\tilde{C}}\rightarrow0.\]
Tensor this sequence by $\rho^{*}\mathcal{G}.$
\[\xymatrix@C=0.6cm{ Tor^1(\rho^{*}\mathcal{G},\mathcal{O}_{\tilde{C}})\ar[r]^{\tau~~~}&\mathcal{O}_{E_i}(-1)^{\oplus_{i=1}^{l-1}}\otimes\rho^{*}\mathcal{G}\ar[r]&\rho^{*}\mathcal{G}\ar[r]&\mathcal{O}_{\tilde{C}}\otimes\rho^{*}\mathcal{G}\ar[r]&0.}\]
$c_1(\mathcal{O}_{\tilde{C}}\otimes\rho^{*}\mathcal{G})=\tilde{L},$ so $c_1($im$\tau)=0,$  while im$\tau$ (i.e. the image of $\tau$) is contained in $\mathcal{O}_{E_i}(-1)^{\oplus_{i=1}^{l-1}}\otimes\rho^{*}\mathcal{G}=\mathcal{O}_{E_i}(-1)^{\oplus_{i=1}^{l-1}},$  which is pure on its support.  Therefore $\tau=0.$  Hence we have
\[0\rightarrow\mathcal{O}_{E_i}(-1)^{\oplus_{i=1}^{l-1}}\rightarrow\rho^{*}\mathcal{G}\rightarrow\mathcal{O}_{\tilde{C}}\otimes\rho^{*}\mathcal{G}\rightarrow0.\]
Push it forward.  Because of the vanishing of $\rho_{*}\mathcal{O}_{E_i}(-1)$ and $R^1\rho_{*}\mathcal{O}_{E_i}(-1),$  we have $\rho_{*}(\rho^{*}\mathcal{G})\simeq \rho_{*}(\mathcal{O}_{\tilde{C}}\otimes\rho^{*}\mathcal{G}).$  

$\rho$ restricted on $\tilde{C}$ is an isomorphism.  So if $\rho_{*}(\rho^{*}\mathcal{G})\simeq \mathcal{G},$  then $\mathcal{O}_{\tilde{C}}\otimes\rho^{*}\mathcal{G}$ is a pure sheaf of rank 1 on $\tilde{C}$ and of Euler characteristic 0,  hence $[\mathcal{O}_{\tilde{C}}\otimes\rho^{*}\mathcal{G}]\in \tilde{\mchi},$  and hence we have found $[\tilde{\mathcal{G}}]=[\mathcal{O}_{\tilde{C}}\otimes\rho^{*}\mathcal{G}]\in \tilde{\mchi},$  such that $f([\tilde{\mathcal{G}}])=[\mathcal{G}].$  

Now we only need to show $\rho_{*}(\rho^{*}\mathcal{G})\simeq \mathcal{G}.$  Firstly,  we show that $\rho_{*}(\rho^{*}\mathcal{O}_{C})\simeq \mathcal{O}_C.$  This can be seen from $\rho_{*}(\rho^{*}\mathcal{G})\simeq \rho_{*}(\mathcal{O}_{\tilde{C}}\otimes\rho^{*}\mathcal{G}),$ with $\mathcal{G}=\mathcal{O}_{C}.$   Then since $\mathcal{G}$ is locally free on its support outside the singular points,   we have that the isomorphism holds outside the singular points;  but around the singular points,  $\rho$ is an isomorphism.   

Finally let $\chi=0$.  The claim on the determinant line bundles is somehow obvious:  by the universal property of $\z,$  we have $\underline{f}^{*}(\z)\simeq (det~R^{\bullet}p~\mathcal{U})^{\vee},$  where $\mathcal{U}$ is the flat family on $X\times \tilde{M}$ obtained by pushing $\tilde{\mathcal{U}}$ forward along $\rho\times id_{\tilde{M}}.$
\[\xymatrix{
\tilde{\mathcal{U}} \ar[d]^{(\rho\times id_{\tilde{M}})_{*}}\ar[r] 
&\tilde{X}\times \tilde{M}\ar[d]^{\rho\times id_{\tilde{M}}}\\
\mathcal{U}\ar[r] &X\times \tilde{M}\ar[d]^{p}\\ & \tilde{M}}\]
Hence $R^{\bullet}p~\mathcal{U}\simeq R^{\bullet}p~((\rho\times id_{\tilde{M}})_{*}\tilde{\mathcal{U}}).$   
\begin{lemma}$R^i(\rho\times id_{\tilde{M}})_{*}\tilde{\mathcal{U}}=0,$ for all $i>0.$
\end{lemma}
\begin{proof}One can see that $\rho\times id_{\tilde{M}}$ is an isomorphism when restricted to the support of $\tilde{\mathcal{U}},$  hence the lemma.
\end{proof}
As $R^i(\rho\times id_{\tilde{M}})_{*}\tilde{\mathcal{U}}=0,$ for all $i>0,$  we have $\underline{f}^{*}\z=det~R^{\bullet}p~\mathcal{U}\simeq det~R^{\bullet}(p\circ (\rho\times id_{\tilde{M}}))~\tilde{\mathcal{U}}= \tilde{\z}.$  Hence $\underline{f_{*}}(\tilde{\Theta}^r)\simeq\underline{f_{*}}(\underline{f}^{*}(\z^r))\simeq\underline{f}_{*}(\mathcal{O}_{\tilde{M}})\otimes \z^r\simeq (j\circ s)_{*}\mathcal{O}_{M^T}\otimes\z^r$ and $f_{*}\tilde{\z}^r\simeq (j\circ s)^{*}\z^r$ for all $r$.  
So we have proven the proposition.
\end{proof}

\begin{rem}According to Proposition \ref{muse},  $\tilde{M}_{\chi}$ is a smooth projective scheme of dimension $g_L+1.$  But Ext$^2(\mf,\mf)_0$ may not vanish for $[\mf]\in \tilde{M},$  because $(\tilde{L},\tilde{K})$ might not satisfy $(\ha'_1).$
\end{rem}
\begin{rem}For the moduli space $\tilde{\mchi}$,  we did not specify the ample line bundle $\mo_{\tilde{X}}(1)$ on the blow-up $\tilde{X},$  but it is easy to see that the moduli space $\tilde{\mchi}$ does not depend on the polarization.
\end{rem}
\begin{prop}\label{ndchi}$\tilde{\mchi}$ is isomorphic to $\tilde{M}$ for any $\chi\in\mathbb{Z}$.
\end{prop}
\begin{proof}Recall that $\tilde{\mchi}$ is a fine moduli space for any $\chi$.  Let $\tilde{\mathcal{U}}_{\chi}$ be some universal sheaf on $\tilde{X}\times \tilde{\mchi}$.  We have the diagram    
\[\xymatrix{
  \tilde{\mathcal{U}}_{\chi}  \ar[r]
                & \tilde{X}\times \tilde{\mchi} \ar[ld]^{q} \ar[d]^{p_{\chi}}  \\
                \tilde{X} &\tilde{\mchi}             }\]

Then $\tilde{\mathcal{U}}_{\chi}\otimes q^{*}\mo_{\tilde{X}}((-\chi)E_1)$ is a flat family on $\tilde{X}\times \tilde{\mchi}$ of stable sheaves of class $\tilde{u}$,  and hence induces a morphism $\varphi_{\chi}:\tilde{\mchi}\ra \tilde{M}$.  It is easy to see that $\varphi_{\chi}$ is bijective,  hence an isomorphism since both $\tilde{\mchi}$ and $\tilde{M}$ are smooth.  Notice that one can construct the isomorphism $\varphi_{\chi}$in many ways and there is no canonical way if $l\geq2$. 
\end{proof}
Now we have identified $(\tilde{M},\tilde{\z}^r)$ with $(M^T,\z^r|_{M^T})$,  hence we can focus on $\tilde{\pi}_{*}\tilde{\z}^r$ on $\tilde{\ls}$,  instead of $\pi^T_{*}(\z^r|_{M^T})$ on $T.$  

\begin{lemma}\label{hzeta}$(1)$ $R^{i}\tilde{\pi}_{*}\tilde{\z}^r=0$ for all $i>0$ and $r>0$,  $R^{i}\tilde{\pi}_{*}\tilde{\z}^r=0$ for all $i<g_L$ and $r<0$;

$(2)$ For $r>0$,  $\tilde{\pi}_{*}\tilde{\z}^r$ is locally free of rank $r^{g_L}$ and $\tilde{\pi}_{*}\tilde{\z}\simeq \mo_{\tilde{\ls}};$

$(3)$ For $r<0$,  $R^{g_L}\tilde{\pi}_{*}\tilde{\z}^r$ is locally free of rank $(-r)^{g_L}$.
\end{lemma}
\begin{proof}By Proposition 3.0.4 in \cite{yuan} we know that $\tilde{\z}(s)$ is ample for $s\gg0$,  hence $\tilde{\z}$ restricted to every fiber of $\tilde{\pi}$ is ample.  By Corollary \ref{tangent} that we will prove later,  the dualizing sheaf on every fiber of $\tilde{\pi}$ is invertible and corresponds to a torsion class in the Picard group.  Hence restricted to every fiber $\tilde{\z}^r$ has no higher cohomology for $r>0$.  Hence $R^{i}\tilde{\pi}_{*}\tilde{\z}^r=0$ for all $i>0$ and $r>0$ and $\tilde{\pi}_{*}\tilde{\z}^r$ is locally free.  Moreover by the basic theory of Jacobians,  we know that $\tilde{\pi}_{*}\tilde{\z}^r$ is of rank $r^{g_L}$.  When $r=1$,  $\tilde{\pi}_{*}\tilde{\z}$ is a line bundle with a nowhere vanishing section hence isomorphic to $ \mo_{\tilde{\ls}}$.

The argument for $r<0$ is analogous.
 \end{proof}
\begin{proof}[Proof of Theorem \ref{thmone}]From the result in \cite{adv},  we know that 
\[Y^1(t)=\sum_{n}h^0(M(c^1_n),\lambda_{u})t^n=\frac{1}{(1-t)^{l+1}}.\]
Then Theorem \ref{thmone} is just a corollary of the Statement 2 in Lemma \ref{hzeta}.
\end{proof}
We obtain the moduli space $\tilde{M}$ by blowing up $l-1$ generic points $x_1,\ldots,x_{l-1}$ on $X.$  On the other hand we may first blow up one point $x_1$ to get a surface $X_1$ with the morphism $\rho_1:X_1\ra X$,  and let $L_1=\rho_1^{*}L-E_1$.  Then similarly we have the moduli space $M_1$ and $\z_1$ which is the determinant line bundle associated to $\mo_{X_1}$.  Tautologically,  blowing up the $l-1$ points $x_1,\ldots,x_{l-1}$ in $X$ is the same as blowing up $\rho_1(x_2),\ldots,\rho(x_{l-1})$ in $X_1$.  Hence we get the same triple ($\tilde{X}$,$\tilde{M}$,$\tilde{\z}$) for both ($X$,$M$,$\z$) and ($X_1$,$M_1$,$\z_1$).  There is a rational map $\nu:M_1--> M,$  but not necessary a morphism in general.  However because of Proposition \ref{muse},  we have the following trivial remark.  Notice that if $L$ satisfies condition $(\ha'_2)$,  then so does $L_1$ for $x_1$ generic.  And $K.L=K_1.L_1-1$ with $K_1=\rho_1^{*}K+E_1$ the canonical divisor on $X_1$.

\begin{rem}Let ($X$,$M$,$\z$),  ($X_1$,$M_1$,$\z_1$) and ($\tilde{X}$,$\tilde{M}$,$\tilde{\z}$) be as in the previous paragraph. Let $T$ be the projective line in $\ls$ defined by asking curves to pass through all the $l-1$ points $x_1,\ldots,x_{l-1}$,  and $T_1$ the line in $|L_1|$ consisting of curves passing through all the $l-2$ points $\rho_1(x_2),\ldots,\rho_1(x_{l-1})$.  If $L$ satisfies $(\ha'_1)$ and $L.K<-1$,  then we have the following Cartesian diagram with $f$ and $f_1$ isomorphisms and $f^{*}\z^r\simeq f_1^{*}\z_1^r\simeq\tilde{\z}^r.$ 
\begin{equation}\label{carbu}\xymatrix{
M_1\ar[d]_{\pi_{1}} &M_1^{T_1}\ar[d]_{\pi^{T_1}_{1}}\ar[l]_{j_1\circ s_1}&\tilde{M} \ar[l]_{f_1}\ar[d]^{\tilde{\pi}_{\chi}}  \ar[r]^{f}
                & M^T\ar[d]^{\pi^T}\ar[r]^{j\circ s}&M\ar[d]^{\pi}  \\
                |L_1|&T_1\ar[l]^{i_1\circ t_1}&|\tilde{L}|\ar[l]^{\imath_1}\ar[r]_{\imath}&T\ar[r]_{i\circ t} &\ls            }\end{equation}
For $\mchi$ with any $\chi$,  we have an analogous Cartesian diagram as $(\ref{carbu})$.
\end{rem}


At the end of this section,  we prove some lemmas which will be used in the next two sections.  Let ($X$,  $L$) and ($\tilde{X}$,  $\tilde{L}$) be the same as in Proposition \ref{muse}.  $K$ and $\tilde{K}$ are the canonical divisor on $X$ and $\tilde{X}$ respectively,  and $\tilde{K}=\rho^{*}K+E_1+\ldots+E_{l-1}.$  Since there is more than one integral curve in $\ls$,  $(\ha'_1)$ implies that $K$ is not effective,  hence nor is $\tilde{K}.$    

\begin{lemma}\label{none}
$h^1(\tilde{L})=h^1(L)=0,$  $h^2(\tilde{L})=h^2(L)=0,$  hence $\chi(L)=l+1$ and $\chi(\tilde{L})=2.$ 
\end{lemma}
\begin{proof}Since $K$ is noneffective,  $L^{-1}\otimes K$ must be noneffective which means $h^0(L^{-1}\otimes K)=h^2(L)=0.$  Similarly $h^2(\tilde{L})$ must be zero because $\tilde{K}$ is not effective.  By a direct computation we get $\chi(L)-\chi(\tilde{L})=h^0(L)-h^0(\tilde{L})=l-1,$  hence $h^1(L)=h^1(\tilde{L})$.

On $X$ we have the following exact sequence
\[0\rightarrow L^{-1}\otimes K\rightarrow K\rightarrow \mathcal{O}_C(K)\rightarrow0,\]
with $C$ some smooth curve in $\ls.$  $L.K<0$,  hence $\mo_{C}(K)$ is locally free on $C$ with negative degree and has no sections.  So there is an injective map sending $H^1(L^{-1}\otimes K)$ into $H^1(K).$  So $h^1(L)=h^1(L^{-1}\otimes K)\leq h^1(K).$  $X$ is simply connected,  then $H^1(K)=0$ and $h^1(L)=0.$  Hence the lemma.  
\end{proof}
\begin{lemma}\label{firstchern} Let $\omega_{\mchi^{IC}}$ denote the canonical line bundle of $\mchi^{IC}$,  then we have $c_1(\omega_{\mchi^{IC}})=[(\pi_{\chi}^{IC})^{*}\mo_{\ls^{IC}}(1)^{\otimes L.K}].$  
\end{lemma}
\begin{proof}The proof is essentially the same as what Danila does in \cite{nila} for $X=\mathbb{P}^2$.  $\mchi^{IC}$ is smooth.  Hence it will suffice to prove that $c_1(\mathcal{T}_{\mchi^{IC}})=[(\pi_{\chi}^{IC})^{*}\mo_{\ls}(-1)^{\otimes L.K}]$,  where $\mathcal{T}_{\mchi^{IC}}$ is the tangent bundle on $\mchi^{IC}.$
   
Recall there is a morphism $\phi_{\chi}^{IC}:\Omega_{\chi}^{IC}\ra \mchi^{IC}$ which is a principal $G$-bundle with $G=\emph{PGL}(V)$.  We have $Pic~(\mchi^{IC})\simeq Pic^G(\Omega_{\chi}^{IC})$ (Theorem 4.2.16 in \cite{dan}).  And also because there is no surjective homomorphism from $G$ to $\mathbb{G}_m,$  the natural morphism $Pic^G(\Omega_{\chi}^{IC})\ra Pic (\Omega_{\chi}^{IC})$ is injective (\cite{git} Chap 1,  Section 3,  Proposition 1.4).  Hence it is enough to prove that $(\phi_{\chi}^{IC})^{*}(c_1(\mathcal{T}_{\mchi^{IC}}))=[(\phi_{\chi}^{IC})^{*}(\pi_{\chi}^{IC})^{*}\mo_{\ls}(-1)^{\otimes L.K}]$

We have a universal sheaf on $X\times\Omega_{\chi}^{IC}.$  We denote it $\mathcal{E}_{\chi}^{IC}$. 
\begin{equation}\label{ppq}
\xymatrix{
  \mathcal{E}_{\chi}^{IC}  \ar[r]
                & X\times \Omega_{\chi}^{IC} \ar[ld]^{q} \ar[d]^{p_{\chi}}  \\
                X & \Omega_{\chi}^{IC}  \ar[d]^{\phi^{IC}_{\chi}}         \\ 
                &\mchi^{IC}\ar[d]^{\pi_{\chi}^{IC}}\\&\ls^{IC} }
\end{equation}
In the Grothendieck group,  we have
\[(\phi_{\chi}^{IC})^{*}\mathcal{T}_{\mchi^{IC}}=\mathcal{E}xt_{p_{\chi}}^1(\mathcal{E}_{\chi}^{IC},\mathcal{E}_{\chi}^{IC}).\]  
And $(\phi_{\chi}^{IC})^{*}(c_1(\mathcal{T}_{\mchi^{IC}}))=c_1((\phi_{\chi}^{IC})^{*}\mathcal{T}_{\mchi^{IC}}).$  So it is enough to compute $c_1((\phi_{\chi}^{IC})^{*}\mathcal{T}_{\mchi^{IC}}).$

Because of $(\ha'_1),$  we have that over every closed point $y\in \Omega_{\chi}^{IC},$  Ext$^i((\mathcal{E}_{\chi}^{IC})_y, (\mathcal{E}_{\chi}^{IC})_y)=0,$ for all $i\geq 2.$  Hence $\mathcal{E}xt_{p_{\chi}}^i(\mathcal{E}_{\chi}^{IC},\mathcal{E}_{\chi}^{IC})=0,$  for all $i\geq2,$  because fiberwise they are Ext$^i((\mathcal{E}_{\chi}^{IC})_y,(\mathcal{E}_{\chi}^{IC})_y).$  Also we have Ext$^0((\mathcal{E}_{\chi}^{IC})_y, (\mathcal{E}_{\chi}^{IC})_y)=\mathbb{C},$  hence $\mathcal{E}xt_{p_{\chi}}^0(\mathcal{E}_{\chi}^{IC},\mathcal{E}_{\chi}^{IC})=(p_{\chi})_{*}\mathcal{H}om(\mathcal{E}_{\chi}^{IC},\mathcal{E}_{\chi}^{IC})$ is a line bundle on $\Omega_{\chi}^{IC},$  hence isomorphic to $\mathcal{O}_{\Omega_{\chi}^{IC}}$ since it has a nowhere vanishing global section.  Therefore
\[ [det~\mathcal{E}xt_{p_{\chi}}^{\bullet}(\mathcal{E}_{\chi}^{IC},\mathcal{E}_{\chi}^{IC})]=[det~R^{\bullet}p_{\chi}~(\mathcal{E}xt^{\bullet}(\mathcal{E}_{\chi}^{IC},\mathcal{E}_{\chi}^{IC}))]=[(det~\mathcal{E}xt_{p_{\chi}}^1(\mathcal{E}_{\chi}^{IC},\mathcal{E}_{\chi}^{IC}))^{\vee}].\]
 Hence 
\begin{equation}\label{dualone}c_1((\phi_{\chi}^{IC})^{*}\mathcal{T}_{\mchi^{IC}})=-c_1(det~R^{\bullet}p_{\chi}~(\mathcal{E}xt^{\bullet}(\mathcal{E}_{\chi}^{IC},\mathcal{E}_{\chi}^{IC}))=-c_1(R^{\bullet}p_{\chi}~(\mathcal{E}xt^{\bullet}(\mathcal{E}_{\chi}^{IC},\mathcal{E}_{\chi}^{IC})).\end{equation}

By Grothendieck-Riemann-Roch,
\[ch(R^{\bullet}p_{\chi}~\mathcal{E}xt^{\bullet}(\mathcal{E}_{\chi}^{IC},\mathcal{E}_{\chi}^{IC}))=(p_{\chi})_{*}(ch(\mathcal{E}_{\chi}^{IC})\cdot ch((\mathcal{E}_{\chi}^{IC})^{\vee})\cdot td(q^{*}\mathcal{T}_{X})),\]
where $\mathcal{T}_{X}$ is the tangent sheaf on $X.$

Since $\mathcal{E}_{\chi}^{IC}$ is a torsion sheaf on $X\times\Omega_{\chi}^{IC},$  
\begin{equation}\label{dualtwo}c_1(R^{\bullet}p_{\chi}~\mathcal{E}xt^{\bullet}(\mathcal{E}_{\chi}^{IC},\mathcal{E}_{\chi}^{IC}))=(p_{\chi})_{*}(-\frac12c_1(\mathcal{E}_{\chi}^{IC})c_1(\mathcal{E}_{\chi}^{IC})c_1(q^{*}\mathcal{T}_{X}))=(p_{\chi})_{*}(\frac12c_1(\mathcal{E}_{\chi}^{IC})^2c_1(q^{*}K)).\end{equation}

$c_1(\mathcal{E}_{\chi}^{IC})$ is just the support of $\mathcal{E}_{\chi}^{IC},$  which is the pull back along $id_{X}\times (\pi_{\chi}^{IC}\circ\phi_{\chi}^{IC})$ of the universal curve in $X\times \ls^{IC}.$  Therefore,   $c_1(\mathcal{E}_{\chi}^{IC})=q^{*}L\otimes p_{\chi}^{*}F,$  where $F$ is the fiber class of $\pi_{\chi}^{IC}$ in $\Omega_{\chi}^{IC},$  i.e. $\mo_{\Omega_{\chi}^{IC}}(F)\simeq(\phi_{\chi}^{IC})^{*}\circ(\pi_{\chi}^{IC})^{*}\mo_{\ls}(1).$ Since $q^{*}L.q^{*}L.q^{*}K=0,$  so we have 
\[\frac12 (c_1(\mathcal{E}_{\chi}^{IC}))^2.(q^{*}K)=q^{*}L.q^{*}K.p^{*}F+\frac12q^{*}K.(p_{\chi}^{*}F)^2.\]
and also $(p_{\chi})_{*}(q^{*}K.(p_{\chi}^{*}F)^2)=0,$  so
\begin{eqnarray}
(p_{\chi})_{*}(\frac12 (c_1(\mathcal{E}_{\chi}^{IC}))^2.(q^{*}K))&=&(p_{\chi})_{*}(q^{*}L.q^{*}K.p_{\chi}^{*}F)\nonumber\\&=&(L.K)F.\nonumber
\end{eqnarray} 
Hence together with (\ref{dualone}) and (\ref{dualtwo}) we have 
\[c_1((\phi_{\chi}^{IC})^{*}\mathcal{T}_{\mchi^{IC}})= [(\phi_{\chi}^{IC})^{*}(\pi_{\chi}^{IC})^{*}\mo_{\ls^{IC}}(-1)^{\otimes L.K}].\]  
Hence the lemma.
\end{proof}

\begin{coro}\label{tangent}$c_1(\mathcal{T}_{\tilde{M}})= [\tilde{\pi}^{*}\mo_{|\tilde{L}|}(-1)^{\otimes (g_L-2)}],$  where $\mathcal{T}_{\tilde{M}}$ is the tangent bundle on $\tilde{M}.$  
\end{coro}
\begin{proof}Since $\tilde{M}$ is smooth,  $c_1(\mathcal{T}_{\tilde{M}})=-c_1(\omega_{\tilde{M}}),$  where $\omega_{\tilde{M}}$ is the canonical line bundle on $\tilde{M}.$  Moreover as stated in Proposition \ref{muse},   $\omega_{\tilde{M}}=f^{*}\omega_{M^T}$.   Because $M^T$ is a complete intersection of $l-1$ divisors in $|\pi^{*}\mo_{\ls}(1)|$ in $M^{IC}$ and also because of Lemma \ref{firstchern},  we have $c_1(\omega_{M^T})=[(\pi^T)^{*}\mo_T(L.K+l-1)]$ and hence $c_1(\omega_{\tilde{M}})=[f^{*}(\pi^T)^{*}\mo_T(L.K+l-1)]=[\tilde{\pi}^{*}\mo_{|\tilde{L}|}(L.K+l-1)].$  Since $L.K+l-1=g_L-2+h^1(L)-h^1(K)=g_L-2,$  we have the lemma. 
\end{proof}

\section{Splitting type for genus one case.}
From now on we are always working on $\tilde{M}.$  So for simplicity,  we drop all the $~\widetilde{}~$ and just write $X,$  $L,$  $M,$  $\z^r,$  $\pi$,  etc.

Now $M$ is a flat family of Jacobians over $\ls\simeq\pone$.  We will give the formulas for $g_L=1,2$ by giving the explicit splitting types for all $\pi_{*}\z^r,$  $r>0.$  By Lemma 3.0.1 in \cite{yuan},  there is a natural global section of $\Theta$ which vanishes at $[\mathcal{F}]\in M$ such that $H^0(\mathcal{F})\neq0.$  Let $D_{\z}=\{[\mf]\in M:h^0(\mf)\neq0\}$ be the divisor associated to that section.

We prove the following proposition in this section.  The technique we use is essentially the same as that in \cite{yuan} for genus one case.
\begin{prop}\label{gone}
If $g_L=1$,   then for $r\geq2,$
\[\pi_{*}\Theta^r\simeq \mathcal{O}_{\ls}\oplus(\mathcal{O}_{\ls}(-i))^{\oplus_{i=2}^r}.\] 
\end{prop}
\begin{proof}
In $X\times\ls\simeq X\times \mathbb{P}^1$,  there is a universal curve $\mathcal{C}$ such that every fiber $\mathcal{C}_s$ is just the curve represented by $s\in\ls.$ 
\[\xymatrix{
  \mathcal{C}  \ar[r]
                & X\times \ls \ar[ld]^{q} \ar[d]^{p}  \\
                X &\ls             }\]
Since $\mc_s$ is integral of genus one,  $\mo_{\mc_s}$ is stable of Euler characteristic zero for every $s$.  Hence the structure sheaf $\mo_{\mc}$ of $\mathcal{C}$ induces an injective morphism embedding $\ls$ as a subscheme of  $M.$
\[\imath:\ls\rightarrow M.\]
It is easy to see that $\imath$ provides a section of the projection $\pi.$  The image of $\imath$ is contained in $D_{\z}$,  and moreover we have the following lemma.
\begin{lemma}$\pi$ restricted to $D_{\z}$ is an isomorphism and $\imath$ is its inverse.
\end{lemma}
\begin{proof}Let $[\mf]\in M$,  and $C$ its support.  Since $C$ is integral and of genus one,  we have $H^0(\mathcal{F})\neq0\Leftrightarrow \mathcal{F}\simeq\mathcal{O}_{C}$.  Hence $D_{\z}$ intersects every fiber of $\pi$ at only one reduced point.  Hence $\pi$ restricted on it is a morphism of degree $1$,  hence an isomorphism.  It is obvious to have $\imath\cdot\pi=id_{\ls}.$ 
\end{proof}

Thus on $M$ we have
\[0\rightarrow\Theta^{-1}\rightarrow\mathcal{O}_{M}\rightarrow\mathcal{O}_{D_{\z}}\rightarrow0.\]

Tensoring by $\Theta^r$ with $r\geq2$,  we get
\begin{equation}\label{odgone}0\rightarrow\Theta^{r-1}\rightarrow\Theta^r\rightarrow\mathcal{O}_{D_{\z}}(\Theta^r)\rightarrow0.\end{equation}

$R^1\pi_{*}\Theta^{r-1}=0$ by Lemma \ref{hzeta}.  Push (\ref{odgone}) forward via $\pi$ and we have
\begin{equation}\label{zonp}
0\rightarrow\pi_{*}\Theta^{r-1}\rightarrow\pi_{*}\Theta^r\rightarrow\pi_{*}\mathcal{O}_{D_{\z}}(\Theta^r)\rightarrow0.\end{equation}

Since $D_{\z}\simeq\ls$ and $\pi\cdot\imath=id_{\ls},$  $\pi_{*}\mathcal{O}_{D_{\z}}(\Theta^r)\simeq \pi_{*}\imath_{*}\imath^{*}\Theta^r\simeq\imath^{*}\z^r.$ 

According to the universal property of $\Theta$,  we have $\imath^{*}\Theta^r\simeq (det(R^{\bullet}p~[\mathcal{O}_{\mathcal{C}}]))^{-r}$.  

We have an exact sequence on $X\times\ls$.
\[0\rightarrow q^{*}\mathcal{O}_{X}(-L)\otimes p^{*}\mathcal{O}_{\ls}(-1)\rightarrow\mathcal{O}_{X\times\ls}\rightarrow\mathcal{O}_{\mathcal{C}}\rightarrow0.\]

Hence $ (det(R^{\bullet}p~[\mathcal{O}_{\mathcal{C}}]))^{-1}\simeq  (det(R^{\bullet}p~[\mathcal{O}_{X\times\ls}]))^{-1}\otimes det(R^{\bullet}p~[q^{*}\mathcal{O}_{X}(-L)\otimes p^{*}\mathcal{O}_{\ls}(-1)]).$  

And also $det(R^{\bullet}p~[\mathcal{O}_{X\times\ls}])\simeq \mathcal{O}_{\ls};$  $det(R^{\bullet}p~[q^{*}\mathcal{O}_{X}(-L)\otimes p^{*}\mathcal{O}_{\ls}(-1)])\simeq \mathcal{O}_{\ls}(-1)^{\otimes \chi(\mathcal{O}_{X}(-L))}.$ 

Since $g_L=1$,  $\chi(\mathcal{O}_{X}(-L))=\chi(\mathcal{O}_{X})=1$ and $\mathcal{O}_{\ls}(\Theta^r)\simeq \mathcal{O}_{\ls}(-r).$

The exact sequence (\ref{zonp}) splits for every $r>1$.  And by induction we get 
\[\pi_{*}\Theta^r\simeq \mathcal{O}_{\ls}\oplus\mathcal{O}_{\ls}(-i)^{\oplus_{i=2}^r}.\]
\end{proof}

In this case,  the generating function can be written down as
\begin{eqnarray}Z^r(t)&=&\sum_{n}h^0(M,\lcn)t^n\nonumber\\
&=&\sum_{n}h^0(M,\Theta^r\otimes\pi^{*}\mathcal{O}_{\ls}(n))t^n\nonumber
\\
&=&\sum_{n}h^0(\ls, \pi_{*}(\Theta^r)\otimes\mathcal{O}_{\ls}(n))t^n\nonumber\\
&=&\large{\frac{1+t^{2}+t^{3}+\ldots+t^{r}}{(1-t)^2}}.\nonumber
\end{eqnarray}
\begin{rem}This result is compatible with Statement 2 in Theorem 4.4.1 in \cite{yuan} as $X=\mathbb{P}^2$ and $L=3H$ or $X=\mathbb{P}(\mo_{\pone}\oplus\mo_{\pone}(-e))$ and $L=2G+(e+2)F$ with $e=0,1.$
\end{rem}
\begin{proof}[Proof of Theorem \ref{thmtwo}]
Recall that we denote
\[Y^r_{g_L=1}(t)=\sum_{n\geq0}y^r_{n,g_L=1}t^n=\frac{1+t^{2}+t^{3}+\ldots+t^{r}}{(1-t)^2};\]
and let $y_{n,g_L=1}^r=0$ for all $n<0$.  In this case we have 
\[Y^r_{g_L=1}(t)=\frac{Z^r(t)}{(1-t)^{l-1}},\]  
hence Theorem \ref{thmtwo}. 
\end{proof}

\section{Splitting type for genus two case.}
Remember that we get the one-dimensional linear system $\ls$ by blowing up $l-1$ points.  So we can write $L=L'-E_1-\ldots-E_{l-1}$ with $L'$ effective and $E_i.L=1$.  We in addition ask $l\geq3.$  Then we have the following proposition.

\begin{prop}\label{gtwo}
For the one-dimensional linear system $L=L'-E_1-\ldots-E_{l-1}$ with $g_L=2$,  if $l-1\geq 2$,  then 

$(1)$ $\pi_{*}\Theta^{r-1}$ is a direct summand of $\pi_{*}\Theta^{r}.$  Let $\pi_{*}\Theta^r=\pi_{*}\Theta^{r-1}\oplus\Delta_r;$

$(2)$ $\pi_{*}\Theta^2\simeq \mathcal{O}_{\ls}\oplus(\mathcal{O}_{\ls}(-2))^{\oplus^3},$  $\pi_{*}\Theta^3\simeq\mathcal{O}_{\ls}(-4)\oplus(\mathcal{O}_{\ls}(-3)^{\oplus^4})\oplus(\mathcal{O}_{\ls}(-2)^{\oplus^3})\oplus\mathcal{O}_{\ls};$  

$(3)$ for $r\geq4,$  we have the recursion formula
\[\pi_{*}\Theta^r\simeq \pi_{*}\Theta^{r-1}\oplus(\mathcal{O}_{\ls}(-r)^{\oplus^2})\oplus(\mathcal{O}_{\ls}(-r-1)^{\oplus^2})\oplus(\Delta_{r-2}\otimes\mathcal{O}_{\ls}(-2)).\] 
\end{prop}

Before proving Proposition \ref{gtwo},   we show some lemmas. 

\begin{lemma}\label{inter}Let $\mathcal{T}$ be the tangent bundle on $M,$  let $c_i(\mathcal{T})$ be its i-th Chern class,  then $c_1(\mathcal{T}).c_i(\mathcal{T})=0$ for all $i.$
\end{lemma}
\begin{proof}According to Corollary \ref{tangent} we have $c_1(\mathcal{T})=[\pi^{*}\mo_{\ls}(-1)^{\otimes (g_L-2)}].$  Denote $F$ to be the fiber class of $\pi$.  It is enough to show that $c_i(\mathcal{T})|_F=0.$  On the other hand,  we can choose a representative of $F$ isomorphic to the Jacobian of some smooth curve.  The tangent bundles on Jacobians are trivial with all Chern classes to be zero.  Hence the lemma.
\end{proof}

Since $h^0(\z)=1,$  we have only one $\z$-divisor $D_{\z}$.  Let $M_1=D_{\Theta}$.  We have exact sequences on $M$.
\begin{equation}\label{a}0\rightarrow\Theta^{-1}\rightarrow\mathcal{O}_{M}\rightarrow\mathcal{O}_{M_1}\rightarrow0.
\end{equation}
\begin{equation}\label{b}0\rightarrow\mathcal{O}_{M}\rightarrow\Theta\rightarrow\mathcal{O}_{M_1}(\Theta)\rightarrow0.
\end{equation}
\begin{equation}\label{c}~~~~~~~~~~0\rightarrow\Theta^{r-1}\rightarrow\Theta^{r}\rightarrow\mathcal{O}_{M_1}(\Theta^{r})\rightarrow0,~~~~r\geq2.
\end{equation}

Pushing (\ref{a}) forward,   we get three isomorphisms of bundles on $\ls.$
\begin{equation}\label{aa}0\rightarrow\pi_{*}\mathcal{O}_{M}\rightarrow\pi_{*}\mathcal{O}_{M_1}\rightarrow0.
\end{equation}
\begin{equation}\label{ab}0\rightarrow R^1\pi_{*}\mathcal{O}_{M}\rightarrow R^1\pi_{*}\mathcal{O}_{M_1}\rightarrow0.
\end{equation}
\begin{equation}\label{ac}0\rightarrow R^2\pi_{*}\Theta^{-1}\rightarrow R^2\pi_{*}\mathcal{O}_{M}\rightarrow0.
\end{equation}
The isomorphism in (\ref{aa}) is because $\pi_{*}\z^{-1}=R^1\pi_{*}\z^{-1}=0$ by Lemma \ref{hzeta}.  The morphism in (\ref{ac}) at first is a surjective map because the relative dimension of $M_1$ over $\ls$ is $1$ and hence $R^2\pi_{*}\mathcal{O}_{M_1}=0;$  then it is an isomorphism because $R^2\pi_{*}\Theta^{-1}$ is a line bundle and $R^2\pi_{*}\mathcal{O}_{M}$ is locally free of rank $1$ on the open set of smooth curves in $\ls$.  And then the morphism in (\ref{ab}) has to be an isomorphism because both (\ref{aa}) and (\ref{ac}) are.

By pushing forward sequence (\ref{b}),   we get three isomorphisms of bundles on $\ls.$
\begin{equation}\label{ba}0\rightarrow\pi_{*}\mathcal{O}_{M}\rightarrow\pi_{*}\Theta\rightarrow0.
\end{equation}
\begin{equation}\label{bb}0\rightarrow \pi_{*}\mathcal{O}_{M_1}(\Theta)\rightarrow R^1\pi_{*}\mathcal{O}_{M}\rightarrow0.
\end{equation}
\begin{equation}\label{bc}0\rightarrow R^1\pi_{*}\mathcal{O}_{M_1}(\Theta)\rightarrow R^2\pi_{*}\mathcal{O}_{M}\rightarrow0.
\end{equation}

We have an isomorphism in (\ref{ba}) because they both are line bundles isomorphic to $\mo_{\ls}$,  (\ref{bb}) and (\ref{bc}) are because $R^j\pi_{*}\Theta^i=0,$  for all $j, ~i>0.$  So we have the following lemma.

\begin{lemma}\label{va}
On $\ls,$  we have

$(1)$ $\pi_{*}\mathcal{O}_{M}\simeq\pi_{*}\Theta\simeq\pi_{*}\mathcal{O}_{M_1}\simeq\mathcal{O}_{\ls};$
  
$(2)$ $R^2\pi_{*}\Theta^{-1}\simeq R^2\pi_{*}\mathcal{O}_{M}\simeq R^1\pi_{*}\mathcal{O}_{M_1}(\Theta)\simeq\mathcal{O}_{\ls}(-2)$.

$(3)$ $R^1\pi_{*}\mathcal{O}_{M}\simeq R^1\pi_{*}\mathcal{O}_{M_1}\simeq\pi_{*}\mathcal{O}_{M_1}(\Theta),$  and they are of rank 2 and Euler characteristic $0$.

$(4)$ $R^1\pi_{*}\mathcal{O}_{M_1}(\Theta^i)=0,$  for all $i\geq2.$

\end{lemma}
\begin{proof}Statement $1$ is trivial.

For statement $2$:  remember that $\z$ restricted to a generic fiber is the usual $\theta$-bundle on the Jacobian by Lemma 3.0.1 in \cite{yuan},  and hence we have $(D_{\z})^g.F=g!.$  By Corollary \ref{tangent} we know that $c_1(\mathcal{T}_M)=0$ since $g_L=2.$  Hence by Hirzebruch-Riemann-Roch,  we have $\chi(\Theta)=-\chi(\Theta^{-1}).$   On the other hand we know that $\chi(\Theta)=\sum(-1)^i\chi (R^i\pi_{*}\Theta)=\chi(\pi_{*}\Theta)=1.$  So as a result $\chi(\Theta^{-1})=\chi(R^2\pi_{*}\Theta^{-1})=-1,$  so the statement.

For statement $3$:  from Lemma \ref{inter} and Hirzebruch-Riemann-Roch we know that $\chi(\mathcal{O}_{M})=c_1(\mathcal{T}).c_2(\mathcal{T})=0,$  hence $\chi(R^1\pi_{*}\mathcal{O}_{M})=\chi(\pi_{*}\mathcal{O}_{M})+\chi (R^2\pi_{*}\mathcal{O}_{M})=0.$

At last we push (\ref{c}) forward and get $R^1\pi_{*}\mathcal{O}_{M_1}(\Theta^r)=0,$  for $r\geq2$.  
\end{proof}
Push (\ref{c}) forward and we get an exact sequence of bundles on $\ls.$
\begin{equation}\label{d}0\rightarrow\pi_{*}\Theta^{r-1}\rightarrow\pi_{*}\Theta^{r}\rightarrow\pi_{*}\mathcal{O}_{M_1}(\Theta^{r})\rightarrow0.~~~for~r\geq2.
\end{equation}

We have already seen that $\pi_{*}\Theta\simeq\mathcal{O}_{\ls}$.  To get the recursion formula,  it is enough to compute the splitting type of $\pi_{*}\mathcal{O}_{M_1}(\Theta^{r})$ for all $r\geq2.$  

We define two other determinant line bundles associated to $\mathcal{O}_{X}(E_2-E_1)$ and  $\mathcal{O}_{X}(E_1-E_2)$ on $X$ respectively.  Let $\eta_1=\lambda_{[\mathcal{O}_{X}(E_2-E_1)]}$ and $\eta_2=\lambda_{[\mathcal{O}_{X}(E_1-E_2)]}.$  According to Lemma 3.0.1 in \cite{yuan},  there is a natural global section of $\eta_1$ (resp. $\eta_2$) whose vanishing locus consists of all $[\mathcal{F}]$ such that $H^0(\mathcal{F}\otimes\mathcal{O}_{X}(E_2-E_1))\neq0$ (resp. $H^0(\mathcal{F}\otimes\mathcal{O}_{X}(E_1-E_2))\neq0$).  We denote the two divisors associated to those two natural global sections as $D_1$ and $D_2$ respectively.

\begin{rem}\label{etath} 
Since $[\mo_X(E_1-E_2)]+[\mo_X(E_2-E_1)]=2[\mo_X]-2[\mo_{pt}],$ we have $\eta_1\otimes\eta_2\simeq \Theta^2(2)$ on $M.$  
\end{rem}

Let $\Pi:=D_1\cap M_1$ and $\Sigma:=D_2\cap M_1.$ 

Now let $\mathcal{C}$ be the universal curve in $X\times \ls$ and $q$ the projection from $X\times\ls$ to $X$.  Then $\mo_{\mathcal{C}}\otimes q^{*}\mo_X(E_1)$ is a flat family of sheaves over $\ls$ and induces a morphism from $\ls$ to $M$ which is a section of $\pi.$  The image of this morphism,  we denote it $\Pi_1,$  is contained in $\Pi=D_1\cap M_1.$  And let $\Pi_2=\overline{\Pi-\Pi_1}.$  We define similarly $\Sigma_1$ and $\Sigma_2$:  $\Sigma_1$ is the image of $\ls$ via the morphism induced by the flat family $\mo_{\mathcal{C}}\otimes q^{*}\mo_X(E_2)$ on $X\times\ls$,  and $\Sigma_2:=\overline{\Sigma-\Sigma_1}$.

Both $\Pi_1$ and $\Sigma_1$ are isomorphic to $\ls\simeq\mathbb{P}^1$.  $\Pi_1\cap\Sigma_1=\emptyset$ because $E_1$ and $E_2$ intersect every curve in $\ls$ at two different points.  For $\Pi_2$ and $\Sigma_2,$  we have the following lemma.
\begin{lemma}$\Pi_2$ is also isomorphic to $\ls$ and provides a section of $\pi$ as well.  The same is true for $\Sigma_2.$   
\end{lemma} 
\begin{proof}Because $E_1$ and $E_2$ do not intersect each other,  they intersect every curve at two different points.  And because curves in $\ls$ are of genus $2$,  any two different points are not linearly equivalent.  So for $i=1,2$,  $\eta_i$ restricted to a fiber is algebraically but not linearly equivalent to the usual $\theta$-bundle.  Moreover according to basic theory of Jacobians,  we know that the intersection number of $\Pi$ with a fiber of $\pi$ is 2.   

So $\pi$ is a morphism of degree $2$ and when restricted on $\overline{\Pi-\Pi_1}$ it is a morphism of degree $1$ over $\mathbb{P}^1$,  hence an isomorphism.  So $\Pi_2=\overline{\Pi-\Pi_1}$ is isomorphic to $\ls$ and provides a section of $\pi.$  It is analogous for $\Sigma_2.$
\end{proof}

Let $C$ be any curve in $\ls$.  We denote $p_C^i$ the point where $E_i$ meets $C.$  $C$ is smooth at $p_C^i$.  For any point $q_C^1\in C,$ such that $h^0(q^1_C-p_C^1+p_C^2)\neq0,$  i.e. $[\mo_C(q^1_C)]\in\Pi,$  there is another point $q_C^2\in C$ satisfying that $q^1_C+p_C^2$ is linearly equivalent to $p_C^1+q_C^2$ on $C.$  Hence if $p^2_C\neq q^2_C,q^1_C\neq p_C^1,$  then $h^0(q^1_C+p^2_C)\geq2.$  And hence by Riemann-Roch,  we know that $h^1(q^1_C+p^2_C)=h^0(\omega_C-q^1_C-p^2_C)\geq 1,$  and hence $\omega_C\sim q^1_C+p^2_C$ since $C$ is of genus $2$ and the canonical sheaf $\omega_C$ on $C$ is of degree $2.$  So $q^1_C$ has either to be $p_C^1$ or satisfies that $\omega_C\sim p^2_C+q^1_C.$  And if $q^1_C=p^1_C$,  then we have $q^2_C=p^2_C$ and $\omega_C\sim p^1_C+p^2_C$.  Hence we can assume that $q^1_C\neq p^1_C$ for a generic $C$,  and hence $\Pi_1\neq \Pi_2$,  $\Sigma_1\neq\Sigma_2$. 

Hence we can specify the universal sheaf on $X\times \Pi_2$ (resp. $X\times\Sigma_2$) as $\mo_{\mathcal{C}}\otimes q^{*}\mo_X(K+L-E_2)$ (resp. $\mo_{\mathcal{C}}\otimes q^{*}\mo_X(K+L-E_1)$).  This is because $\mo_C(K+L)\simeq\omega_C$ for all $[C]\in\ls,$  and $\omega_C\sim p^2_C+q^1_C$ which implies that $\mo_C(K+L-E_2)\sim\mo_C(q^1_C).$
\begin{lemma}\label{deg}For $i=1,2$ we have
$\pi_{*}(\Theta^r|_{\Pi_i})\simeq\mathcal{O}_{\ls}(-r\chi(\mo_X))=\mo_{\ls}(-r),$  which is equivalent to saying that $D_{\z}.\Pi_i=-1.$  And the same holds for $\Sigma_i,$  $i=1,2.$
\end{lemma}
\begin{proof}By the universal property of $\Theta$ we have that $\z|_{\Pi_1}=(det~R^{\bullet}p~\mathcal{U}^1)^{-1}$ where $\mathcal{U}^1\simeq \mo_{\mathcal{C}}\otimes q^{*}\mo_X(E_1) $ is the universal sheaf on $X\times\Pi_1$.  And also we have the exact sequence on $X\times \ls:$
\[0\rightarrow p^{*}\mathcal{O}_{\ls}(-1)\otimes q^{*}\mathcal{O}_{X}(-L+E_1)\rightarrow q^{*}\mathcal{O}_{X}(E_1)\rightarrow \mathcal{U}_1\rightarrow0.\]
So 
\[det~R^{\bullet}p~\mathcal{U}_1\simeq det~R^{\bullet}p~(q^{*}\mathcal{O}_{X}(E_1))\otimes (det~R^{\bullet}p~(p^{*}\mathcal{O}_{\ls}(-1)\otimes q^{*}\mathcal{O}_{X}(-L+E_1)))^{-1}.\]
Then we have 
\[ det~R^{\bullet}p~(q^{*}\mathcal{O}_{X}(E_1))\simeq \mo_{\ls},\]
\[ det~R^{\bullet}p~(p^{*}\mathcal{O}_{\ls}(-1)\otimes q^{*}\mathcal{O}_{X}(-L+E_1))\simeq \mo_{\ls}(-1)^{\otimes \chi(\mo_X(-L+E_1))}.\] 

$\chi(\mo_X(-L+E_1))=\chi(\mo_X(E_1))-\chi(\mo_C(E_1))=\chi(\mo_X(E_1))$,  since $C$ is a curve of genus $2$ and $\mo_{C}(E_1)$ is a line bundle of degree $1$ on $C.$  By Hirzebruch-Riemann-Roch we know that $\chi(\mo_X(E_1))=\chi(\mo_X)=1.$

For $\Pi_2$,  we use $\mo_{\mathcal{C}}\otimes q^{*}\mo_X(K+L-E_2)$ as the universal sheaf.  Similar computation shows that $D_{\z}.\Pi_2=-\chi(\mo_X(K+L-E_2))=-\chi(\mo_X)$ since $K.(K+L)=2g_L-2=2.$  

For $\Sigma_i$ the argument is analogous. 
\end{proof}
$\Pi+\Sigma\sim(2D_{\z}+2F)|_{D_{\z}}.$  Lemma \ref{deg} implies that $(\Pi+\Sigma).D_{\z}=-4$.  Moreover $F.D^2_{\z}=g!=2,$  hence we have $2D_{\z}^3+4=(\Pi+\Sigma).D_{\z}=-4.$  Then we get the following proposition immediately.
\begin{prop}\label{alpha}On the moduli space $M,$  we have $D_{\z}^3=-4.$  
\end{prop}

Since we know that $\chi(\z)=1$,  by Proposition \ref{alpha} we can compute $\chi(\z^r(n))$ for all $r$ and $n$.  And we have 
\begin{equation}\label{rnchi}
\chi(\z^r(n))=-\frac23r^3+nr^2+\frac53r.
\end{equation}

However,  if we want to write down explicitly the splitting type of $\pi_{*}\z^r$ and get a result which is not only numerical but also gives some geometric description,  we have to see how the four projective lines,  $\Pi_1$,  $\Pi_2$,  $\Sigma_1$ and $\Sigma_2$  intersect each other.  It is obvious that $\Pi_1\cap\Sigma_1=\emptyset$ because $E_1$ and $E_2$ intersect every curve in $\ls$ at two different points.   We have several lemmas: 

\begin{lemma}\label{pszero}$\Pi_2$ has no intersection with $\Sigma_2,$ i.e. $\Pi_2.\Sigma_2=0$.
\end{lemma}
\begin{proof}
Let $C$ be any curve in $\ls$.  As we mentioned before,  if $[\mo_C(q_C^1)]\in\Pi_2$ and $[\mo_C(q_C^2)]\in\Sigma_2$,  then $q_C^1+p_C^2\sim p_C^1+q_C^2$ with $p^i_C$ the point where $C$ meets $E_i.$  Since $p^1_C\neq p^2_C,$  and $p_C^1-p_C^2\sim q_C^1-q_C^2,$  we have $q^1_C\neq q^2_C$ for any $[C]\in\ls$ and hence the lemma.
\end{proof}
Now we compute $\Pi_1.\Sigma$ and $\Pi.\Sigma_1.$

Notice that the universal sheaf $\mathcal{U}^1$ over $X\times \Pi_1$ can be chosen to be $\mathcal{O}_{\mathcal{C}}\otimes q^{*}\mathcal{O}_{X}(E_1),$  as a result $[\mathcal{F}]\in \Pi_1\cap \Sigma\Leftrightarrow H^0(\mathcal{O}_{C_{\mathcal{F}}}\otimes q^{*}\mo_X(E_1)\otimes q^{*}\mo_X(E_1-E_2))\neq0,$  where $C_{\mf}$ is the supporting curve of $\mf.$  It is analogous for $\Pi\cap\Sigma_1.$  

Let $\mathcal{B}^1=\mathcal{O}_{\mathcal{C}}\otimes q^{*}\mathcal{O}_X(2E_1-E_2),$  $\mathcal{B}^2=\mathcal{O}_{\mathcal{C}}\otimes q^{*}\mathcal{O}_X(2E_2-E_1).$  These two  sheaves are also flat families over $X\times\ls$ hence induce two embeddings mapping $\ls$ to $M$ which both are sections of $\pi.$  Denote their image in $M$ as $P_1$ and $P_2$ respectively.  $P_i\simeq\pone.$

\begin{lemma}\label{degp}$\Theta|_{P_i}\simeq \mathcal{O}_{\mathbb{P}^1}(-\chi(\mo_X)+2)=\mo_{\mathbb{P}^1}(1),$  for $i=1,2.$
\end{lemma}
\begin{proof}The proof is analogous to Lemma \ref{deg},  and instead of $\chi(-L+E_1)$ we have $\chi(-L+2E_1-E_2)$ or $\chi(-L+2E_2-E_1)$ which are equal to $\chi(-L+E_1)-2.$
\end{proof}

\begin{lemma}\label{dm}
For any curve $\mathbf{C}$ in $M,$  let $d=deg~\Theta|_{\mathbf{C}},$ 

$(1)$ If $d<0$,  then $\mathbf{C}\subset M_1.$

$(2)$ If $d\geq 0,$  and also $\mathbf{C}$ is not contained in $M_1,$  then $d=\#(\mathbf{C}\cap M_1),$  counting with multiplicity. 
\end{lemma}
\begin{proof}If the curve is not contained in $M_1=D_{\z},$  then there is a nonzero global section of $\Theta$ vanishing at points corresponding to sheaves with global sections.  Hence the degree of $\Theta$ restricted to that curve should be nonnegative and must equal to $\mathbf{C}\cap M_1$ counting with multiplicity. 
\end{proof}

\begin{rem}\label{contain}Because of Lemma \ref{dm},  if $P_1$ (resp. $P_2$) is not contained in $M_1,$  then $\Pi_1.\Sigma=\# \Pi_1\cap \Sigma=1$ (resp. $\Pi.\Sigma_1=\# \Pi\cap \Sigma_1=1$).  
\end{rem}

\begin{lemma}\label{notcon}Neither $P_1$ nor $P_2$ is contained in $M_1.$ 
\end{lemma}
\begin{proof}Note that a priori,  $P_i$ is contained in $D_i$ for $i=1,2.$  If $P_1$ is contained in $M_1$,  then $P_1\subset M_1\cap D_1=\Pi$.  Hence $P_1$ has to be either $\Pi_1$ or $\Pi_2$.  But $\z$ restricted on $P_1$ has degree $1$ while restricted on $\Pi_i$ it has degree $-1$ by Lemma \ref{deg}.  So we know that $P_1$ can not be contained in $M_1$.  For $P_2$ it is analogous. 
\end{proof}

Because of Lemma \ref{notcon} and Remark \ref{contain},  we have $\Pi_1.\Sigma=\Pi.\Sigma_1=1$.  On the other hand,  we have $\Pi_1\cap\Sigma_1=\emptyset$,  $\Pi_2\cap\Sigma_2=\emptyset$.  Hence we have $\Pi_1.\Sigma_2=1$ and $\Pi_2.\Sigma_1=1.$   We now only need to compute $\Pi_1.\Pi_2$ and $\Sigma_1.\Sigma_2.$   

Recall that $D_{\z}=M_1.$   Now on $M_1$ we have an exact sequence.
\begin{equation}\label{me}
0\rightarrow\eta_1^{-1}\otimes\eta_2^{-1}\rightarrow\mathcal{O}_{M_1}\rightarrow\mathcal{O}_{M_2}\rightarrow0.
\end{equation}
$M_2$ is a subscheme of $M_1,$  which equals to $\Pi+\Sigma$ as a divisor.  $\Pi+\Sigma\sim (2D_{\z}+2F)|_{D_{\z}}.$  
Because of Remark \ref{etath} we can rewrite sequence (\ref{me}) as follows:
\begin{equation}\label{vme}
0\rightarrow\z^{-2}(-2)|_{M_1}\rightarrow\mathcal{O}_{M_1}\rightarrow\mathcal{O}_{M_2}\rightarrow0.
\end{equation}
Using formula (\ref{rnchi}),  by a direct computation we get $\chi(\mo_{M_2})=2$.  Hence the arithmetic genus of $M_2$ is negative.  Also we know that $M_2=\Pi_1+\Pi_2+\Sigma_1+\Sigma_2,$  and the $\Pi_i$ and the $\Sigma_i$ are isomorphic to $\pone.$  So $M_2$ can not be connected and therefore $\Pi_1\cap\Pi_2=\Sigma_1\cap\Sigma_2=\emptyset$.

\begin{rem}\label{rzero}So the picture of these four curves is very clear:
$\Pi_1\cap\Pi_2=\emptyset=\Sigma_1\cap\Sigma_2;$  $\Pi_1.\Sigma_2=1$ and $\Pi_2.\Sigma_1=1;$  and $\Pi_1\cap\Sigma_1=\Pi_2\cap\Sigma_2=\emptyset.$ \end{rem}   

We have the exact sequence on $M_2$ as follows.
\begin{equation}\label{ssmr}
0\rightarrow(\mathcal{O}_{\Pi_1}(-1)\oplus\mathcal{O}_{\Pi_2}(-1))\otimes\Theta^r\rightarrow\mathcal{O}_{M_2}(\Theta^r)\rightarrow(\mathcal{O}_{\Sigma_1}\oplus\mathcal{O}_{\Sigma_2})\otimes\Theta^r\rightarrow0
\end{equation}

We then have the following proposition.
\begin{prop}\label{th}
$\pi_{*}\mathcal{O}_{M_2}(\Theta^r)\simeq \mathcal{O}_{\ls}(-1-r)^{\oplus^2}\oplus\mathcal{O}_{\ls}(-r)^{\oplus^2}.$
\end{prop}
\begin{proof}
By Lemma \ref{deg} we have $\pi_{*}(\Theta^r|_{\Pi_i})\simeq\pi_{*}(\Theta^r|_{\Sigma_i})\simeq  \mathcal{O}_{\ls}(-r),$  for $i=1,2.$  So push (\ref{ssmr}) forward and we get
\begin{equation}\label{mrp}
0\rightarrow\mathcal{O}_{\ls}(-1-r)^{\oplus^2}\rightarrow\pi_{*}\mathcal{O}_{M_2}(\Theta^r)\rightarrow\mathcal{O}_{\ls}(-r)^{\oplus^2}\rightarrow0
\end{equation}
It is easy to see there are no higher direct image along $\pi$ for sheaves on $M_2$,  since $\pi$ restricted on $M_2$ has relative dimension zero.  And sequence (\ref{mrp}) splits for every $r$.
\end{proof}

We tensor the sequence (\ref{vme}) by some power of $\Theta.$  Then we have following exact sequences on $M_1$.
\begin{equation}\label{mea}
0\rightarrow\mo_{M_1}(\Theta^{-2}(-2))\rightarrow\mathcal{O}_{M_1}\rightarrow\mathcal{O}_{M_2}\rightarrow0.
\end{equation}
\begin{equation}\label{meb}
0\rightarrow\mo_{M_1}(\Theta^{-1}(-2))\rightarrow\mathcal{O}_{M_1}(\Theta)\rightarrow\mathcal{O}_{M_2}(\Theta)\rightarrow0.
\end{equation}
\begin{equation}\label{mes}
0\rightarrow\mathcal{O}_{M_1}(\Theta^{r-2}(-2))\rightarrow\mathcal{O}_{M_1}(\Theta^{r})\rightarrow\mathcal{O}_{M_2}(\Theta^{r})\rightarrow0,~~~~r\geq0.
\end{equation}

Push all of them forward and we get
\begin{equation}\label{meea}
0\rightarrow\pi_{*}\mathcal{O}_{M_1}\rightarrow\pi_{*}\mathcal{O}_{M_2}\rightarrow R^1\pi_{*}\mathcal{O}_{M_1}(\Theta^{-2})\otimes\mathcal{O}_{\ls}(-2)\rightarrow R^1\pi_{*}\mathcal{O}_{M_1}\rightarrow0.
\end{equation}
\begin{equation}\label{meeb}
0\rightarrow \pi_{*}\mathcal{O}_{M_1}(\Theta)\rightarrow \pi_{*}\mathcal{O}_{M_2}(\Theta)\rightarrow R^1\pi_{*}\mathcal{O}_{M_1}(\Theta^{-1})\otimes\mathcal{O}_{\ls}(-2)\rightarrow R^1\pi_{*}\mathcal{O}_{M_1}(\Theta)\rightarrow0.
\end{equation}
\begin{equation}\label{mees}
\small{\xymatrix@C=0.2cm{0\ar[r]&\pi_{*}\mathcal{O}_{M_1}(\Theta^{r-2})\otimes\mathcal{O}_{\ls}(-2)\ar[r]&\pi_{*}\mathcal{O}_{M_1}(\Theta^{r})\ar[r]&\pi_{*}\mathcal{O}_{M_2}(\Theta^{r})\ar[r] &R^1\pi_{*}\mathcal{O}_{M_1}(\Theta^{r-2})\otimes\mathcal{O}_{\ls}(-2)\ar[r]&0,  r\geq2.}}
\end{equation}

In (\ref{meea}) and (\ref{meeb}),  the zeros on the right are because $R^1\pi_{*}\mo_{M_2}(\z^r)=0$ for all $r$.  The left zeros are because  $\pi_{*}\mathcal{O}_{M_1}(\Theta^{-r})=0,~\forall r\geq1.$ In (\ref{mees}) the right zero is because $R^1\pi_{*}\mathcal{O}_{M_1}(\Theta^{r})=0$ as $r\geq2$ by Lemma \ref{va}.  And (\ref{mees}) will be a short exact sequence with three terms when $r\geq4.$  Then we have a simple corollary of Proposition \ref{th}.

\begin{coro}\label{trivial}The canonical sheaf $\omega_{M}$ on $M$ is trivial.
\end{coro}
\begin{proof}Since by Corollary \ref{tangent} we already know that $c_1(\mathcal{T}_{M})=0,$  it is enough to show $h^0(\omega_{M})=h^3(\mathcal{O}_{M})=1.$  

From Proposition \ref{th} and Statement $3$ in Lemma \ref{va} and also sequence (\ref{meeb}), we can see that $\chi(\pi_{*}\mo_{M_1}(\z))=0,$  and there is a injective morphism from $\pi_{*}\mo_{M_1}(\z)$ to $\pi_{*}\mo_{M_2}(\z)\simeq\mo_{\ls}(-1)^{\oplus2}\oplus\mo_{\ls}(-2)^{\oplus2}.$  Hence $\pi_{*}\mathcal{O}_{M_1}(\Theta)\simeq\mathcal{O}_{\ls}(-1)^{\oplus^2}.$  Also according to Lemma \ref{va},  we have $\pi_{*}\mathcal{O}_{M_1}(\Theta)\simeq R^1\pi_{*}\mathcal{O}_{M_1}\simeq R^1\pi_{*}\mathcal{O}_{M}\simeq\mathcal{O}_{\ls}(-1)^{\oplus^2},$  and $\pi_{*}\mathcal{O}_{M_1}\simeq\mathcal{O}_{\ls}.$  Hence $H^1(R^1\pi_{*}\mo_{M_1})=H^2(\pi_{*}\mo_{M_1})=0$.  On the other hand,  since $\pi$ restricted on $M_1$ is of relative dimension $1$,  we have $R^i\pi_{*}\mo_{M_1}=0$ for all $i\geq2$.  Hence by the spectral sequence we know that $H^2(\mo_{M_1})=0$.

From sequence (\ref{a}) we have the exact sequence as follows
\[H^2(\mo_{M_1})\ra H^3(\z^{-1})\ra H^3(\mo_{M})\ra0.\]
Because $R^2\pi_{*}\z^{-1}\simeq\mathcal{O}_{\ls}(-2)$ and $R^i\pi_{*}\z^{-1}=0$ for all $i<2$,  we have $h^3(\z^{-1})=h^1(R^2\pi_{*}\z^{-1})=1;$  together with the vanishing of $H^2(\mo_{M_1})$,  we get $h^3(\mathcal{O}_{M})=h^3(\z^{-1})=1.$    
\end{proof}
Corollary \ref{trivial} gives us an interesting result in the theory of compactified Jacobians of integral curves with planar singularities as follows.
\begin{coro}\label{jaco}Let $X$ be any simply connected smooth projective surface over $\mathbb{C}$,  $L$ be an effective line bundle satisfying $(\ha'_1)$ and $(\ha'_2)$,  moreover $dim~\ls\geq 3$ and $g_L=2$,  then for a generic integral curve $C$ in $\ls$,  the compactified Jacobian $J^{g_L-1}$ which parametrizes the rank one torsion free sheaves of Euler characteristic zero has its dualizing sheaf be the trivial line bundle. 
\end{coro}
\begin{proof}[Proof of Proposition \ref{gtwo}]As stated in the proof of Corollary \ref{trivial},  we already know that $\pi_{*}\mathcal{O}_{M_1}(\Theta)\simeq R^1\pi_{*}\mathcal{O}_{M_1}\simeq\mathcal{O}_{\ls}(-1)^{\oplus^2}.$ We rewrite (\ref{mees}) with $r=2$ as
\begin{equation}\label{vmeec}
0\rightarrow \mathcal{O}_{\ls}(-2)\rightarrow\pi_{*}\mathcal{O}_{M_1}(\Theta^2)\rightarrow\mathcal{O}_{\ls}(-3)^{\oplus^2}\oplus\mathcal{O}_{\ls}(-2)^{\oplus^2}\rightarrow \mathcal{O}_{\ls}(-3)^{\oplus^2}\rightarrow0.
\end{equation}
Hence $\pi_{*}\mathcal{O}_{M_1}(\Theta^2)\simeq\mathcal{O}_{\ls}(-2)^{\oplus^3},$  together with sequence (\ref{d})  we get the expression for $\pi_{*}\Theta^2.$  Lemma \ref{va} also says that $R^1\pi_{*}\mathcal{O}_{M_1}(\Theta)\simeq \mathcal{O}_{\ls}(-2).$  So sequence (\ref{mees}) with $r=3$ implies that $\pi_{*}\mathcal{O}_{M_1}(\Theta^3)\simeq \mathcal{O}_{\ls}(-4)\oplus\mathcal{O}_{\ls}(-3)^{\oplus^4}$.  Then we know the splitting type of $\pi_{*}\Theta^3.$

For $\Theta^r,$ $r\geq 4,$  both (\ref{d}) and (\ref{mees}) are short exact sequences with three terms and split,  which implies Statements $1$ and $3$ in the proposition.
\end{proof}
We have defined $Z^r(t)=\sum_{n}h^0(M,\lcn)t^n
=\sum_{n}h^0(M,\Theta^r\otimes\pi^{*}\mathcal{O}_{\ls}(n))t^n$.  The generating function $Z^r(t)$ can be written down explicitly as follows:
\begin{enumerate}
\item $Z^1(t)=\large{\frac{1}{(1-t)^2}};~Z^2(t)=\large{\frac{1+3t^{2}}{(1-t)^2}};
~Z^3(t)=\large{\frac{1+3t^{2}+4t^{3}+t^4}{(1-t)^2}}.$
\item $for~r\geq 4,~~Z^r(t)=Z^{r-1}(t)+(Z^{r-2}(t)-Z^{r-3}(t))\cdot t^2+\large{\frac{2t^{r}+2t^{r+1}}{(1-t)^2}}.$
\end{enumerate}
The recursion formula 2 implies that 
\[Z^r(t)=\frac{1+3t^{2}+\sum_{i=3}^r ((i+1)t^{i}+(i-2)t^{i+1})}{(1-t)^2}~for~r\geq2.\]
\begin{rem}These results are compatible with Statement 2 in Theorem 4.5.2 in \cite{yuan} as $X=\mathbb{P}(\mo_{\pone}\oplus\mo_{\pone}(-e))$ and $L=2G+(e+3)F$ with $e=0,1.$  
\end{rem}
\begin{proof}[Proof of Theorem \ref{thmthree}]
In this case we have \[Y^r_{g_L=2}(t)=\frac{Z^r(t)}{(1-t)^{l-1}},\] and hence the theorem.
\end{proof} 

\begin{flushleft}{\textbf{Acknowledgments.}}  I would like to thank Lothar G\"ottsche for his guidance and Barbara Fantechi,  Eduardo de Sequeira Esteves and Ramadas Ramakrishnan Trivandrum for many helpful discussions. 
\end{flushleft}


\begin{thebibliography}{99}
\bibitem{new}
Luis \'Alvarez-C\'onsul,  Alastair King,  \emph{A functorial construction of Moduli of Sheaves.}  Invent. math. 168,  613-666(2007).

\bibitem{mofir}
Alina Marian,  Dragos Oprea, \emph{A tour of theta dualities on moduli spaces of sheaves, Curves and abelian varieties,} 175-202, Contemporary Mathematics, 465, American Mathematical Society, Providence, Rhode Island (2008).



\bibitem{nila}
G. Danila,  \emph{R\'esultats sur la conjecture de dualit\'e \'etrange sur le plan projectif.} Bull. Soc. Math. France 130 (2002), 1--33. 

\bibitem{adv}
Geir Ellingsrud, Lothar G\"ottsche, Manfred Lehn,  \emph{On the cobordism class of the Hilbert scheme of a surface}, J. Algebraic Geom. 10 (2001), no. 1, 81--100.



\bibitem{dan}
D. Huybrechts,  M. Lehn,  \emph{The Geometry of Moduli Spaces of Sheaves}.   Friedr.  Vieweg \& Sohn Verlagsgesellschaft mbH,  Braunschweig/Wiesbaden,  1997.


\bibitem{le}
J. Le Potier,  \emph{Faisceaux semi-stables et syst\`emes coh\'erents}. Proceedings de la Conference de Durham (July 1993),  Cambridge University Press (1995),  p.179-239


\bibitem{git}
D. Mumford, \emph{Geometric Invariant Theory.}  Springer-Verlag Berlin-Heidelberg-New York (1965)

\bibitem{yuan}
Y.Yuan,  \emph{Determinant line bundles on Moduli spaces of pure sheaves on rational surfaces and Strange Duality},  arXiv:  1005.3201.
\end{thebibliography}
\end{document}